\documentclass{amsart}
\usepackage{amsmath}
\usepackage{amsfonts}
\usepackage{amsthm, upref}
\usepackage{graphicx}

\newcommand{\comment}[1]{}
\newcommand{\eq}{\begin{equation}}
\newcommand{\en}{\end{equation}}
\newcommand{\rr}{\mathbb{R}}

\begin{document}

\theoremstyle{plain}
\newtheorem{thm}{Theorem}
\newtheorem{lemma}[thm]{Lemma}
\newtheorem{prop}[thm]{Proposition}
\newtheorem{cor}[thm]{Corollary}

\theoremstyle{definition}
\newtheorem{defn}{Definition}
\newtheorem{asmp}{Assumption}
\newtheorem{notn}{Notation}
\newtheorem{prb}{Problem}

\theoremstyle{remark}
\newtheorem{rmk}{Remark}
\newtheorem{exm}{Example}
\newtheorem{clm}{Claim}

\title[sharpness principle]{On a conjectured sharpness principle for probabilistic forecasting with calibration}

\author{Soumik Pal}
\address{Department of Mathematics \\ University of Washington\\ Seattle, WA 98195}
\email{soumik@u.washington.edu}

\begin{abstract}
This note proves a weak type of the sharpness principle as conjectured by Gneiting, Balabdaoui, and Raftery \cite{GBR} in connection with probabilistic forecasting subject to calibration constraints. A strong version of such a principle still awaits a proper formulation.
\end{abstract}




\date{\today}
\maketitle


\section{Introduction} This note is concerned with the conjectured {sharpness principle} described in the Gneiting-Balabdaoui-Raftery article on probabilistic forecasts \cite{GBR} subject to calibration constraints. Using predictive distributions for making forecasts has been steadily gaining acceptance in various fields such as meteorology (Gneiting \& Raftery \cite{GR05}), economics (Granger \cite{granger06}), and finance (Duffie \& Pan \cite{dp97}). The conceptual appeal of probabilistic forecasts and the advances in Markov Chain Monte Carlo methodology have made them very attractive. For a recent editorial on probabilistic forecasting, please see the article by Gneiting \cite{gn08}.

Predictive performances of forecasts have been traditionally measured by what is known as the probability integral transform or PIT. The PIT essentially involves applying the forecasted distribution functions to their corresponding true observations and checking the uniformity of the resulting histogram. This method first proposed by Dawid \cite{dawid84} and further by Diebold, Gunther, \& Tay \cite{DGT98}, only leads to a necessary condition for a forecaster to be ideal. In fact, that using the PIT just by itself can lead to erroneous conclusions was succinctly demonstrated by an example of Hamill \cite{ham01}. To address this issue Gneiting {et al.} in \cite{GBR} introduce the concept of maximizing the sharpness of the predictive distributions subject to calibrations. They define calibration as {\lq}the statistical consistency between the predictive distributions and the associated observations, and is a joint property of the predictions and the values that materialize{\rq}. Sharpness, on the other hand, {\lq}refers to the concentration of the predictive distributions and is a property of the forecasts only{\rq}. The authors then define a theoretical framework to assess calibration and sharpness and distinguish between several modes of calibration. Our article is based on this framework.

Let $\{F_t, \; t=1,2,\ldots\}$ and $\{G_t, \; t=1,2,\ldots \}$ denote sequences of continuous and strictly increasing distribution functions on the real line, possibly depending on stochastic parameters. One thinks of the $G_t$'s as the true data-generating distributions and the corresponding sequence of $F_t$'s as the sequence of probabilistic forecasts. Among all the modes of calibration proposed in \cite{GBR}, probably the most useful notion is that of probabilistic calibration. The sequence $\{F_t\}$ is probabilistically calibrated relative to the sequence $\{G_t\}$ if  
\begin{equation}\label{calibfull}
\lim_{T\rightarrow \infty}\frac{1}{T}\sum_{i=1}^T G_i\circ F_i^{-1}(p) = p , \quad \forall\; p\in (0,1).
\end{equation}
It is not difficult to see that probabilistic calibration is essentially equivalent to the uniformity of the PIT histogram, a fact that leads to the practicality of its applications. 

We should mention that calibration in the context of forecasting binary events is well studied in Statistics literature. See, for example, the articles by DeGroot \& Fienberg \cite{DF82}, Dawid \& Vovk \cite{DV99}, and Shafer \& Vovk \cite{SV01}. The elegant game-theoretic approach of Vovk \& Shafer \cite{VS05} is a natural culmination of this approach. The models considered here, however, are continuous and not within the purview of the previous approach. 

\section{Main results}

The particular problem addressed in this note is the one stated in Section 2.4 of \cite{GBR}. It involves what the authors of \cite{GBR} call the {sharpness principle}. Ideally one would like to forecast the data generating distribution perfectly. Supposing that one maximizes sharpness among predictive distributions subject to calibration, how close to perfect forecasts can she get ? The strong sharpness principle in \cite{GBR} states that perfect forecasts and maximization of sharpness subject to calibration is indeed the same. For probabilistic calibration it is difficult to find natural examples where such a principle holds. Thus, in \cite{GBR}, the authors also formulate the weak sharpness principle. Suppose the sequence $\{F_t\}$ is probabilistically calibrated with respect to $\{G_t\}$. If one assesses sharpness by variance, what can one deduce about the accuracy of the forecasts $F_t$'s ? A partial answer, in the case of {climatological forecaster} (when all the $F_t$'s are the same), is provided in Theorem 1 of \cite{GBR}. It is shown that on average the forecasts will have a bigger variance than the actual observations. A similar conclusion for the general case is conjectured, and this is what we settle here. The choice of averaged variance as a measure of sharpness is somewhat arbitrary. Please see section 3 in Gneiting et al. \cite{GSGHJ} for further discussion.

The main result is the following proposition.

\begin{prop}\label{thmfin} 
Suppose, for some $T\in \mathbb{N}$, the sequence of probability distributions $\{F_i\}$ and $\{G_i\}$ satisfy the finite probabilistic calibration condition 
\begin{equation}\label{calib}
\frac{1}{T}\sum_1^T G_t\circ F_t^{-1}(p) = p, \qquad \text{for all}\; p\in (0,1).
\end{equation}

For all such sequences $(F_1, F_2, \ldots, F_T)$ and $(G_1, G_2, \ldots, G_T)$ satisfying \eqref{calib}, we have
\begin{equation}\label{varineq}
\frac{1}{T}\sum_1^T Var(F_i) \ge \frac{1}{T}\sum_1^T Var(G_i),
\end{equation}
with equality if and only if
\begin{equation}\label{wheneq}
E(G_1)- E(F_1)= E(G_2) - E(F_2)= \ldots=E(G_T) - E(F_T).
\end{equation}
Here $E(H)$ and $Var(H)$ denote the expectation and the variance of a random variable with distribution function $H$.
\end{prop}

With stronger assumptions one can extend the finite horizon result to a limiting result consistent with the definition of probabilistic calibration given in \cite{GBR}. Below we outline one such approach. Our assumptions are deliberately kept stronger than required in order to maintain simplicity of the statement. 

\begin{prop}\label{thm2}
Assume that there is a bounded interval $(a,b)$ such that each $F_i$ is zero to the left of $a$ and one to the right of $b$. Further assume that the sequences $\{F_i\}$ and $\{G_i\}$ satisfy the probabilistic calibration condition \eqref{calibfull} and that there is a function $\theta:(0,1)\rightarrow \rr$ such that
\eq\label{limtheta}
\lim_{T\rightarrow \infty} \sup_{u\in(0,1)}\left\lvert\frac{1}{T} \sum_{i=1}^T \left( F_i^{-1}(u) \right)^2 - \theta(u)\right\rvert=0. 
\en
Then
\begin{equation}\label{fullvarineq}
\liminf_{T\rightarrow \infty} \frac{1}{T}\sum_{i=1}^T Var(F_i) \ge \limsup_{T\rightarrow \infty} \frac{1}{T}\sum_{i=1}^T Var(G_i).
\end{equation}
Here $\limsup$ and $\liminf$ denote the limit superior and the limit inferior of a real valued sequence.
\end{prop}

The significance of the previous inequalities is directly linked to the Gneiting-Balabdaoui-Raftery \cite{GBR} paradigm of maximizing sharpness of a forecast subject to calibration. In practice, competing probabilistic forecasters need to be assessed on the basis of their predictive distributions and the outcomes that materialize, in accordance with Dawid's \cite{dawid84} prequential principle. Ideally we would like a forecaster to predict the true distribution, i.e. $F_t=G_t$ for all times $t$. However, $G_t$ is unobservable, so this is not an operational criterion. Instead \cite{GBR} proposes maximizing sharpness of the predictive distributions subject to calibration. Our results lend further support to their proposal, in the sense that if we measure sharpness by average variance, a probabilistically calibrated forecaster will not be sharper than the ideal forecaster.

However, Proposition \ref{thmfin} allows equality in \eqref{varineq} under apparently weaker conditions \eqref{wheneq} than distributional equality. The quest for notions of calibration that lead to the ideal forecaster as the unique sharpest forecaster (a strong form of the sharpness principle) remains open.

Let us add that one should compare these results to existing results about decomposition of strictly proper scores into sharpness and reliability like terms for both binary forecasting and more general set-up. Please see the recent article by Br\"ocker \cite{brocker} for a complete result and references of previous work. Such a decomposition, of course, will be the most complete version of a sharpness principle that one can desire. However, in our framework, it is not obvious how to obtain a similar decomposition theorem.

The rest of this note contains the proofs of the stated results.

\section{Proofs}

\begin{proof}[Proof of Theorem \ref{thmfin}]  Let $Z=(Z_1, Z_2, \ldots, Z_T)$ be a multinomial random vector with an equal probability of success in each of the $T$ bins.  Let $U$ be the random variable
\[
U:=F_1(X_1)^{Z_1}F_2(X_2)^{Z_2}\ldots F_T(X_T)^{Z_T}
\]
where $X=(X_1,X_2, \ldots, X_T)$ is a random vector such that marginally $X_1 \sim G_1, \ldots, X_T \sim G_T$ and is independent of $Z$. Then, as was mentioned in the paper \cite{GBR}, the finite probabilistic calibration condition implies $U$ is a Uni$(0,1)$ random variable. 

For $(u,z)$ among the possible values of $(U,Z)$, define the function
\eq\label{whatish}
H(z,u)= \left( F_1^{z_1}F_2^{z_2}\ldots F_T^{z_T} \right)^{-1}(u).
\en
We are going to compute $Var(H(Z,U))$ in two different methods.

The first method is to condition on $Z=z$, and use the decomposition
\begin{equation}\label{firstineq}
\begin{split}
Var(H(Z,U))&= E\; Var(H \lvert Z=z) + Var\; E(H \lvert Z=z).  
\end{split}
\end{equation}
Now, it is straightforward to see that
\eq\label{varg}
\begin{split}
Var\left( H\lvert Z=z \right)&= Var(X_i), \quad \text{if}\; z_i=1, z_j=0 \; j \neq i\\
E \left(  H\lvert Z=z  \right)&= E(X_i), \quad \text{if}\; z_i=1, z_j=0 \; j \neq i.
\end{split}
\en
Thus, 
\[
Var(H(Z,U)) = \frac{1}{T}\sum_{i=1}^T Var(X_i) + \frac{1}{T}\sum_{i=1}^T (\mu_i - \bar{\mu})^2.
\]
Here $\mu_i= E (G_i)$ and $\bar{\mu}$ is the mean of the $\mu_i$'s. Thus
\begin{equation}\label{varineq1}
Var(H(Z,U)) \ge \frac{1}{T}\sum_{i=1}^T Var(G_i)
\end{equation}
with equality if and only if all the $G_i$'s have the same mean.

\medskip

The second method to compute the variance of $H(Z,U)$ is to first condition on $U$, and repeat the variance decomposition formula
\begin{equation}\label{decomp2}
\begin{split}
Var(H(Z,U)) &= E\; Var(H \lvert U=p) + Var\; E(H \lvert U=p).
\end{split}
\end{equation}
To compute the above, note that
\begin{equation}\label{computecond2}
\begin{split}
E(H \lvert U=p) &= \frac{1}{T}\sum_{i=1}^T F_i^{-1}(p),\quad E(H^2 \lvert U=p) = \frac{1}{T} \sum_{i=1}^T \left( F_i^{-1}(p) \right)^2,\\
Var(H \lvert U=p) &= \frac{1}{T} \sum_{i=1}^T \left( F_i^{-1}(p) \right)^2 - \frac{1}{T^2}\left(\sum_{i=1}^T F_i^{-1}(p)\right)^2.
\end{split}
\end{equation}
Define $Y_i= F_i^{-1}(U)$, $i=1,2,\ldots,T$. By the finite probabilistic calibration condition, $U$ is marginally uniform. Hence, each $Y_i$ has distribution $F_i$. Now, from the equalities in \eqref{decomp2} and \eqref{computecond2}, we get
\eq\label{vardecomp3}
\begin{split}
Var(H(Z,U))&= \frac{1}{T}\sum_{i=1}^T E \left( Y_i^2 \right) - \frac{1}{T^2} E\left(\sum_{i=1}^T Y_i\right)^2 \\
& + \frac{1}{T^2} Var \left( \sum_{i=1}^T Y_i \right)\\
&= \frac{1}{T}\sum_{i=1}^T E \left( Y_i^2 \right) - \frac{1}{T^2} \left( \sum_{i=1}^T  E(Y_i)\right)^2.
\end{split}
\en
The finally equality follows from the fact that for any $Z$, one has $Var(Z)=E Z^2 - (E Z)^2$.

Thus, we get the equality
\eq\label{varf}
\begin{split}
Var(H(Z,U))&= \frac{1}{T}\sum_{i=1}^T Var(Y_i) + \frac{1}{T}\sum_1^T (E Y_i)^2 - \frac{1}{T^2} \left( \sum_{i=1}^T  E(Y_i)\right)^2,\\
&= \frac{1}{T}\sum_{i=1}^T Var(Y_i) + \frac{1}{T}\sum_{i=1}^T (\alpha_i - \bar{\alpha})^2,
\end{split}
\en
where $\alpha_i = E(F_i)$ and $\bar{\alpha}$ is the mean of the $\alpha_i$'s. 

Thus, if all $\alpha_i$'s (i.e., all the means of $F_i$'s) are the same, we get exact equality above. In that case, combining with inequality \eqref{varineq1}, we get the result
\begin{equation}\label{varineqsame}
\frac{1}{T}\sum_{i=1}^T Var(F_i)  \ge  \frac{1}{T}\sum_{i=1}^T Var(G_i),  
\end{equation}
with equality if and only if the means of all the $G_i$'s are the same.

To prove \eqref{varineq}, in case when all the $\alpha_i$'s are not the same, define
\eq\label{howmeant}
F_i^*(x)= F_i(x+\alpha_i), \quad G_i^*(x)= G_i(x+\alpha_i).
\en
Notice, that the calibration condition \eqref{calib} continues to hold. Additionally, each $F_i^*$ has now mean zero and the same variance as $F_i$. Applying the variance inequality \eqref{varineqsame} for case of $\{F_i^*\}$ and $\{G_i^*\}$, we get the complete result as in the statement.
\end{proof}

\begin{proof}[Proof of Theorem \ref{thm2}]
The proof proceeds by extending the idea of variance decomposition in the finite case to the limit as $T$ tends to infinity. For every fixed $T$, as before, one defines
\eq\label{whatisut}
U(T):=F_1(X_1)^{Z_1}F_2(X_2)^{Z_2}\ldots F_T(X_T)^{Z_T}
\en
where $X=(X_1,X_2, \ldots, X_T)$ is a random vector such that marginally $X_1 \sim G_1, \ldots, X_T \sim G_T$ and is independent of $Z(T)=(Z_1,\ldots,Z_T)$, which is a multinomial with $T$ cells and equal probability of success in each cell. 

Now, the difference here with the finite case is that $U(T)$ is no longer uniform. However, it follows from the condition \eqref{calibfull} that $U(T)$ converges in distribution (or, weak convergence, see Billingsley \cite{billingsley}), as $T$ tends to infinity to $U$, which is a Uni$(0,1)$ random variable. 
However, one can still proceed with a variance decomposition as in the last proof. Notice that equality \eqref{varineq1} continues to hold for each finite $T$ when we define $H$ in \eqref{whatish} by replacing $Z$ by $Z(T)$ and $U$ by $U(T)$.

The equality in \eqref{varf} also holds, but $Y_i$ does not follow $F_i$ since $U(T)$ is not uniform. However, from \eqref{vardecomp3}, one can still write the following. 
\[
Var(H(Z(T),U(T))) \le \frac{1}{T}\sum_{i=1}^T E\left( F_i^{-1}(U(T)\right)^2.
\] 
Combining the observations above, we get that for every fixed $T$, one has
\eq\label{finiteineq1}
E \left[ \frac{1}{T}\sum_{i=1}^T \left( F_i^{-1}(U(T)\right)^2 \right] \ge \frac{1}{T} \sum_{i=1}^T Var(G_i).
\en
\medskip

Now $U(T)$ converges to $U$ in distribution. Define the sequence of functions
\[
\theta_T(u) = \frac{1}{T} \sum_{i=1}^T \left( F_i^{-1}(u) \right)^2, \quad u\in (0,1).
\]
Then, each $\theta_T$ is continuous and bounded by our assumptions. Also, as $T$ tends to infinity, $\theta_T$ converges to $\theta$ uniformly. Thus, from the standard theory of weak convergence, it follows that $\lim_{T\rightarrow \infty} E \theta_T(U(T))$ exists and is given by $E \theta(U)$ where $U$ is Uni$(0,1)$. In particular, from \eqref{finiteineq1}, it follows that
\eq\label{limsupbnd}
E \theta(U) \ge \limsup_{T\rightarrow \infty}\frac{1}{T} \sum_{i=1}^T Var(G_i).
\en

Assume, for now, that each $F_i$ has the same zero mean. Since $\theta_T$ converges to $\theta$ pointwise and remain uniformly bounded, it follows from the Dominated Convergence Theorem that
\eq\label{liminfbnd}
E \theta(U) = \lim_{T\rightarrow \infty} E \theta_T(U) = \frac{1}{T} \sum_{i=1}^T Var(F_i),
\en
since $F_i^{-1}(U)$ has the distribution $F_i$ and we have assumed its mean to be zero. Combining \eqref{limsupbnd} and \eqref{liminfbnd} we have proved the stated proposition.

Finally, the mean zero assumption on $F_i$'s can be easily removed by suitable translations as was shown in the proof of the last proposition.
\end{proof}

\section*{Acknowledgment} Many thanks to Tilmann Gneiting for introducing me to the problem and for the subsequent comments, discussion, and encouragement.

\end{document}